\documentclass[11pt,twoside]{article}
\usepackage{a4}
\usepackage{amssymb,amsmath,amsthm,latexsym}
\usepackage{amsfonts}
  \usepackage{amsfonts}
\usepackage{graphicx}
\usepackage{hyperref}

\newtheorem{theorem}{Theorem}[section]

\newtheorem{corollary}[theorem] {Corollary}
\newtheorem{definition}[theorem]{Definition}

\newtheorem{lemma} [theorem]{Lemma}

\newtheorem{remark}[theorem]{Remark}

\vspace{5cm}

\newcommand{\RR}{{\mathcal{R}}}

\newcommand{\ZZ}{\mathcal {Z}}

\newcommand{\bc}{\begin{center}}
\newcommand{\ec}{\end{center}}
\newcommand{\be}{\begin{enumerate}}
\newcommand{\ee}{\end{enumerate}}
\newcommand{\bi}{\begin{itemize}}
\newcommand{\ei}{\end{itemize}}

\newcommand{\beq}{\begin{equation}}
\newcommand{\eeq}{\end{equation}}
\newtheorem{exmp}{Example}
\newcommand{\bex}{\begin{exmp}}
	\newcommand{\eex}{\end{exmp}}

\newcommand{\seq}[1]{\ensuremath{ \left \{ #1(n): n \in \ZZ \right \}}}

\newcommand{\Sol}[1] {\textbf{Solution:}}

\newcommand {\abs}[1] {\lvert#1\rvert}

\newcommand{\normb}[1]{\ensuremath{\left\lVert#1\right\rVert }}

\newcommand\frightarrow{\scalebox{1}[.3]{$\rule[.45ex]{2ex}{1.5pt}%
		\kern-.2ex{\blacktriangleright}$}}

\newcommand{\vertiii}[1]{{\left\vert\kern-0.25ex\left\vert\kern-0.25ex\left\vert #1 
		\right\vert\kern-0.25ex\right\vert\kern-0.25ex\right\vert}}

\begin{document}
  
  \label{'ubf'}  
\setcounter{page}{1}                                 

\markboth {\hspace*{-9mm} \centerline{\footnotesize \sc
   Relations Between Discrete Maximal Operators in Harmonic Analysis  }
                 }
                { \centerline                           {\footnotesize \sc  
         Anupindi Sri Sakti Swarup AND Michael Alphonse                                                 } \hspace*{-9mm}              
               }


\begin{center}
{ 
       {\Large \textbf { \sc  Relations Between Discrete Maximal Operators in Harmonic Analysis
                               }
       }
\\

\medskip

{\sc Sri Sakti Swarup Anupindi and Michael Alphonse }\\
{\footnotesize Department of Mathematics, Birla Institute of Technology and Science, Jawahar Nagar,Hyderabad-500 078, Telangana, India}\\
{\footnotesize Department of Mathematics, Birla Institute of Technology and Science, Jawahar Nagar,Hyderabad-500 078, Telangana, India
}\\
{\footnotesize e-mail: {\it p20180442@hyderabad.bits-pilani.ac.in}}
{\footnotesize e-mail: {\it  alphonse@hyderabad-bits-pilani.ac.in}}
}
\end{center}

\thispagestyle{empty}

\hrulefill

\begin{abstract}  
{\footnotesize  In this paper, we define several types of maximal operators on sequence spaces occuring in Harmonic analysis and present various connections between them.
}
 \end{abstract}
 \hrulefill



\section{Introduction}\label{sec1}
In this paper, we define several types of maximal operators on real valued sequence spaces and study relationships between them. Several such operators are defined in continuous case and studied in standard literature in harmonic analysis. We present some of them namely, Hardy-Littlewood maximal operator (both centered and non-centered), dyadic maximal operator and sharp maximal opeator. A good-$\lambda$ inequality is presented which relates dyadic maximal operator and sharp maximal operator. For details of these maximal operators on real line, refer to~\cite{Duan1}. \\ \\
The method of Calderon-Zygmund decomposition on sequence spaces plays an important role in studying the relationship between these operators~\cite{Sak1}. In the discrete case, Calderon-Zygmund decomposition uses dyadic intervals. When we study the relation between the maximal operators, we are required to double the intervals that destroys the dyadic nature of the intervals. This challenge is not there in the case of real line~\cite{Duan1}.

 \section{Preliminaries and Notation}

Throughout this paper, $\ZZ$ denotes set of all integers and $\ZZ_{+}$ denotes set of all positive integers. For a given interval $I$ in $\ZZ$ , $\abs{I}$ always denotes the cardinality of $I$. 
 For each positive integer N, consider the collection of disjoint intervals of cardinality $2^N$, 
\[ \left \{ I_{N,j} \right \}_{j \in \ZZ} = \left \{ [ (j-1)2^N+1, \dots, j2^N ] \right \}_{ j \in \ZZ}  \]
The set of all intervals which are of the form $I_{N, j}$ where $ N \in \ZZ_{+}$ and $j \in \ZZ$ are 
called dyadic intervals. For fixed $N$, we denote the set of all intervals  $\left \{ I_{N,j} \right \}_{j \in \ZZ}$ as $\mathcal{I}_N$. For fixed N, the intervals in $\mathcal{I}_N$ are disjoint.
For a dyadic interval I, we define
\begin{align*}
2LI & = [ (j-2)2^N+1, \dots, j2^N] \\
2RI & = [(j-1)2^N+1, \dots, (j+1)2^N] \\
3I & = [(j-2)2^N+1, \dots, (j+1)2^N ] 
\end{align*}
Note $\abs{3I}=3.2^N$. Note that $2LI, 2RI$ are dyadic intervals each of length $2^{N+1}$ but $3I$ is  not an dyadic interval.\\

The Calderon-Zygmund decomposition theorem for sequences~\cite{Sak1} is as follows. 
\begin{theorem}
Let $1\leq p < \infty$ and a $\in \ell^p(\ZZ)$. For every $t>0$, and $0\leq \alpha<1$, there exists a sequence of disjoint dyadic intervals $\left \{ I_j^t \right \}$ such that
\begin{align*}
(i)& \quad t < \frac{1}{ \abs{I_j^t}^{1-\alpha}   } \sum_{k \in I_j^t} \abs{a(k)} \leq 2t,\forall j \in \ZZ \\
(ii)& \quad \forall n  \not\in \cup_j I_j^t, \quad \abs{a(n} \leq t \\
(iii)& \quad \text{If} \quad t_1 > t_2 , \quad \textbf{then each} \quad I_j^{t_1} \quad \textbf{is subinterval of some} \quad I_m^{t_2}, \quad \forall j,m \in \ZZ
\end{align*}
\end{theorem}

An operator $T$ is bounded on  $\ell^p(\ZZ)$ if  $ \forall a \in \ell^p(\ZZ)$
\[ \normb{Ta}_{\ell^p(\ZZ)} \leq C_p \normb{a}_{\ell^p(\ZZ)} \] 

An operator $T$ is of weak type (1,1)  on $\ell^p(\ZZ)$ if for each $a \in \ell_1(\ZZ)$ 
\[ \abs{ \left \{ m \in \ZZ: \abs{Ta(m) } > \lambda \right \} } \leq \frac{C}{\lambda} \normb{a}_1 \]

For $\seq{a} \in \ell^p(\ZZ)$, norm in $\ell^p(\ZZ)$  (refer to ~\cite{Duan1}) is given by
\[ \normb{a}_{\ell^p(\ZZ)} = \int_0^\infty p {\lambda}^{p-1} \abs {\left \{ m \in \ZZ: \abs{a(m)}>\lambda \right \} } d \lambda \]

%



\section{ Definitions}
\subsection{Maximal Operators}
Let $\seq{a}$ be a sequence. We define the following three types of Hardy-Littlewood maximal operators as follows:
\begin{definition}
 If $I_r$ is the interval $\left \{-r,-r+1, \dots, 0, 1, 2, \dots,r-1,r \right\} $, define centered Hardy-Littlewood maximal operator
\[ M^\prime a(m) = \sup_{r >0} \frac{1}{(2r)}  \sum_{n \in I_r} \abs{a(m-n)}    \]
We define Hardy-Littlewood maximal operator as follows

\[ M a(m) = \sup_{m \in I} \frac{1}{\abs{I}} \sum_{ n \in I} \abs{a(n)}   \]
 where the supremum is taken over all intervals containing $m$.

\end{definition}

\begin{definition}
We define dyadic Hardy-Littlewood maximal operator as follows:
\[ M_d a(m) = \sup_{ m \in I} \frac{1}{\abs{I}} \sum_{ k \in I} \abs{a(k)} \]
where supremum is taken over all dyadic intervals containing $m$.

\end{definition}
 Given a sequence $\seq{a}$ and an interval $I$, let $a_I$ denote average of $\textbf{\seq{a}}$ on $I$, i.e $ a_I = \frac{1}{\abs{I}} \sum_{ m \in I} a(m) $. Define the sharp maximal operator $M^{\#}$ as follows  
\[ M^{\#} a(m) = \sup_{ m \in I} \frac{1}{\abs{I}} \sum_{ n \in I} \abs{ a(n) - a_I} \]
where the supremum is taken over all intervals $I$ containing $m$. We say that sequence $\seq{a}$ has bounded mean oscillation if the
sequence $M^{\#}a$ is bounded. The space of sequences with this property is called sequences of bounded mean oscillation and is denoted by BMO($\ZZ$).
We define a norm in BMO($\ZZ$) by $ \normb{a}_\star = \normb{M^{\#}a}_\infty$. The space BMO($\ZZ$) is studied in $\text{\cite{Mich1},\cite{Mich2}}$.

\section{ Relation between Maximal operators }
\begin{theorem}{\label{Sak_lem_2.11}}
Given a sequence $\left \{ a(m): m \in \ZZ \right \}$, the following relation holds:
\[   M^{\prime}a(m)  \leq M a(m) \leq 3 M^\prime a(m) \]
\end{theorem}
\begin{proof}
First inequality is obvious as $M^\prime a$ considers supremum over centered intervals, while $M$ considers supremum over all intervals. 
For second inequality,
let $I = [ m-r_1, m-r_1+1, \dots, m+r_2-1, m+r_2 ]$ containing $m$. Let $ r = max \left\{ r_1, r_2 \right \}$. Consider $I_1 = [ m-r, m-r+1, \dots, m+r-1, m+r ]$ containing $m$. Note that $\abs{I_1}= 2r+1, \abs{I} = r_1+r_2+1$. 
Then
\[ \abs{I} = r_2 + r_1 + 1 \geq r = \frac{1}{3} 3r \geq \frac{1}{3} (2r+1) =  \frac{1}{3} \abs{I_1} \] This gives
\[ \frac{1}{\abs{I}} \sum_{ k \in I} \abs{a(k)} \leq \frac{3}{\abs{I_1}} \sum_{ k \in I_1} a(k) \leq 3 M^\prime a(m) \]
\end{proof}

\begin{theorem}{\label{Sak_lem_2.12}} 
If $\textbf{a}= \left \{ a(k): k \in \ZZ \right \}$ is a sequence with $\textbf{a} \in \ell_1$, then 
\[ \abs{ \left \{  m \in  \ZZ: M^\prime a(m) > 4  \lambda \right \} } \leq 3  \abs{ \left \{ m \in \ZZ: M_d a(m) > \lambda \right \} } \]
\end{theorem}
\begin{proof}
Using Calderon Zygmund decomposition at height $\lambda$, we obtain a collection of disjoint dyadic intervals $\left \{I_j: j \in \ZZ^{+} \right\}$ such that
\[ \lambda < \frac{1}{\abs{I_j}} \sum_{k \in I_j} \abs{a(k)} \leq 2\lambda \] Then
\[  \cup_j I_j \subseteq \left \{ m \in \ZZ: M_d a(m) > \lambda \right \}   \]
It suffices to show that 
\[ \left \{ m \in \ZZ: M^\prime a(m) > 4  \lambda \right \} \subset \cup_j 3I_j \]
Let $m \notin \cup_j 3I_j$. We shall prove $ m \notin \left \{ k \in \ZZ: M^\prime a(k) > 4 \lambda \right \}$.
Let $I$ be any interval centered at $m$. Choose $ N \in \ZZ_{+}$ such that $ 2^{N-1} \leq \abs{I} < 2^N $. Then $I$ intersects exactly 2 dyadic intervals in $\mathcal{I}_N$ say $R_1, R_2$. Assume $R_1$ intersects $I$ on the left and $R_2$ intersects $I$ on the right. Since $m \notin \cup_{j=1}^\infty 3I_j$, $m \notin 2RI_j, j= 1 \dots$ and $ m \notin 2LI_j, j =1,\dots$. But $m \in 2RR_1$ and $m \in 2LR_2$. 

Therefore, both $R_1$ and $R_2$ cannot be any one of $I_j$.

%
%

Hence, the average of $\seq{a}$ on each $R_i, i =1, 2 $ is at most $\lambda$. Further note that $\frac{\abs{R_1}}{\abs{I}} \leq 2, \frac{\abs{R_2}}{\abs{I}} \leq 2 $. So 
\begin{align*}
 \frac{1}{\abs{I}} \sum_{m \in I} \abs{a(m)} & \leq \frac{1}{\abs{I}} \biggl( \sum_{ k \in R_1} \abs{a(k)} + \sum_{ k \in R_2} \abs{a(k)} \biggr) \\
& = \biggl( \frac{1}{\abs{R_1}} \frac{\abs{R_1}}{\abs{I}} \sum_{k \in R_1} \abs{a(k) } + \frac{1}{\abs{R_2}} \frac{\abs{R_2}}{\abs{I}} \sum_{k \in R_2} \abs{a(k) } \biggr)\\
& \leq 2 \biggl( \frac{1}{\abs{R_1}}   \sum_{k \in R_1} \abs{a(k) } + \frac{1}{\abs{R_2}}   \sum_{k \in R_2} \abs{a(k) } \biggr)  = 2(\lambda+ \lambda) = 4 \lambda
\end{align*}
\end{proof}

\begin{corollary}{\label{Sak_cor_1}}
For a sequence $\seq{a}$, if $M_d a \in \ell^p(\ZZ), 1< p< \infty$, then 
\[ \normb{M^\prime a}_{\ell^p(\ZZ)} \leq C \normb{M_d a}_{\ell^p(\ZZ)} \]
\end{corollary}

\begin{proof}
\begin{align*}
\normb{M^\prime a}_{\ell^p(\ZZ)}& = \int_0^\infty p \lambda^{p-1} \abs{ \left \{ m: M^\prime a(m) > \lambda \right \} } d\lambda \\
& \leq 3 (4)^{p-1} \int_0^\infty p ({\frac{\lambda}{4}})^{p-1} \abs{ \left \{ m: M_d a(m) > \frac{\lambda}{4} \right \} } d\lambda  \\
& \leq 3 (4)^{p} \int_0^\infty p u^{p-1} \abs{ \left \{ m: M_d a(m) > u \right \} } du  \\
& \leq 3(4)^p  \normb{M_d a}_{\ell^p(\ZZ)}
\end{align*}

\end{proof}

\begin{remark}{\label{Sak_rem_1}}
From corollary[$\ref{Sak_cor_1}$], whenever $M^\prime$ is of weak type (1,1), then $M_d$ is also of weak type (1,1). 
From Theorem[$\ref{Sak_lem_2.12}$], if $M^\prime$ is of weak type (1,1), then $M_d$ is of weak type (1,1). 
From Theorem[$\ref{Sak_lem_2.11}$], if $M$ is of weak type (1,1), then $M^\prime$ is of weak type (1,1). It is well known that $M$ is of weak  type (1,1), refer to~\cite{Sak1}.
\end{remark}

In the following lemma, we see that in the norm of BMO($\ZZ$) space, we can replace the average $a_I$ of $\left \{ a(n) \right \}$ by a constant $b$. The proof is similar to the proof in continuous version~\cite{Duan1}. We provide the proof for the sake of completeness. 
\begin{lemma}{\label{Sak_Prop_6.5}} 
Consider  a non-negative sequence $\textbf{a}= \left \{ a(k): k \in \ZZ \right \}$. Then the following are valid.
\begin{align*}
1.& \quad \frac{1}{2} \normb{a}_{\star} \leq \sup_{m \in I} \inf_{b \in \RR} \frac{1}{\abs{I}} \abs{ a(m)- b} \leq \normb{a}_\star \\
2.& \quad M^{\#}(\abs{a})(i) \leq M^{\#} a(i) , i \in \ZZ
\end{align*}

\end{lemma}

\begin{proof}

For first inequality, note for all $b \in \RR$, 
\[ \sum_{ m \in I} \abs{a(m) - a_I }  \leq \sum_{ m \in I} \abs{a(m) - b} + \sum_{ m \in I} \abs{b-a_I} = A + B (say)  \]
Now
\begin{align*}
B &= \abs{I} \abs{ b - a_I} = \abs{I} \abs{ b - \frac{1}{\abs{I}} \sum_{k \in I} a(k) } \\
&=\abs{I} \abs{ \frac{1}{\abs{I}} \biggl( \sum_{k \in I}  ( b -a(k)) \biggr) }  \leq \sum_{ k \in I} \abs{ b -a(k) } 
\end{align*}
So,
\[ \sum_{ m \in I} \abs{a(m) - a_I }  \leq \sum_{ m \in I} \abs{a(m) - b} + \sum_{ m \in I} \abs{b-a_I} \leq 2 \sum_{ m \in I} \abs{a(m) - b}   \]

Now, divide both sides by $\abs{I}$, and take infimum over all $b$ followed by, supremum over all $I$. This proves
\[ \quad \frac{1}{2} \normb{a}_{\star} \leq \sup_{m \in I} \inf_{b \in \RR} \frac{1}{\abs{I}} \abs{ a(m)- b}\]
The proof for second inequality
\[ \sup_{m \in I} \inf_{b \in \RR} \frac{1}{\abs{I}} \abs{ a(m)- b} \leq \normb{a}_\star \] is obvious.

The proof of (2) follows from the fact that $\abs { \abs{a} - \abs{b} } \leq \abs{a} - \abs{b} $ for any $ a, b \in \RR$.
%
\end{proof}

\begin{lemma}{\label{Sak_lem_6.10}}
If $a \in \ell^{p_0}(\ZZ) $ for some $p_0$, $ 1 \leq p_0 < \infty$, then for all $ \gamma > 0$ and $ \lambda >0$
\[ \abs { \left \{ n \in \ZZ: M_d a(n) > 2\lambda , M^{\#} a(n) \leq \gamma \lambda \right \} } \leq 2 \gamma \abs{ \left \{ n \in \ZZ: M_d a(n) > \lambda \right \} } \]
\end{lemma}

\begin{proof}
Perform Calderon-Zygmumd decomposition for the sequence $\seq{a}$ at height $\lambda$, which gives collection of intervals $\left \{ I_j \right \}$ such that for each $j$,
\[ \lambda \leq \frac{1}{\abs{I_j}} \sum_{k \in I_j} \abs{a(k)}  \leq 2 \lambda \]
Let $I$ be one of the interval in the collection $\left \{ I_j \right \}$. In Calderon-Zygmund decomposition, there exists interval $\tilde{I}$ such that $\tilde{I}$ is either $2RI$ or $2LI$ and 
\[ \frac{1}{\abs{\tilde{I}}} \sum_{k \in \tilde{I}} \abs{a(k)} \leq \lambda \] For details, refer to section on Preliminaries and Notation.

It is easy to observe that $ \forall m \in I$, $M_d a(m) > 2 \lambda $ implies $ M_d( a \chi_{I} )(m) >2\lambda $. 
Let,
\begin{align*}
a_1 &= ( a- a_{\tilde{I}}) \chi_I \\
a_2 &= a_{\tilde{I}}\chi_I
\end{align*}
Then, since $M_d$ is sublinear
\begin{align*}
a_1+a_2 &= (a-a_{\tilde{I}} ) \chi_I + a_{\tilde{I}} \chi_{I} \\
M_d(a_1+a_2) &\leq M_d( ( a- a_{\tilde{I}}) \chi_I ) + M_d( a_{\tilde{I}} \chi_I)   \\
&\leq M_d( ( a- a_{\tilde{I}}) \chi_I ) +( a_{\tilde{I}}) 
\end{align*}
Since $M_d( a_{\tilde{I}} \chi_I ) (k) \leq a_{\tilde{I}} \quad \forall k$, it follows that
\[M_d(a_1+a_2) = M_d(a\chi_I) \leq M_d( ( a- a_{\tilde{I}}) \chi_I ) +( a_{\tilde{I}}) \]
Hence for every, $k \in I$, it follows that
\[ M_d( ( a- a_{\tilde{I}}) \chi_I ) (k) \geq  M_d(a\chi_I)(k) -a_{\tilde{I}} \]
So, for those $k's$,
\[ M_d (( a-a_{\tilde{I}})\chi_I)(k) \geq M_d (a \chi_I)(k) - a_{\tilde{I}} > \lambda \]
By remark [$\ref{Sak_rem_1}$], using $weak(1,1)$ inequality for $M_d$
\begin{align*}
\left \{ k \in \ZZ: M_d ((a- a_{\tilde{I}}) \chi_I )) (k) > \lambda \right \}  & \leq \frac{C}{\lambda} \sum_I \abs{ a(k) - a_{\tilde{I}} } \\
& \leq \frac{2} { {\lambda}} \abs{I}\frac{C}{\abs{\tilde{I}}} \sum_{\tilde{I}} \abs{a(k) -a_{\tilde{I}} } \\
& \leq \frac{2C} { {\lambda}} \abs{I} \inf_{m \in I} M^{\#} a(m) \\
& \leq \frac{2C}{ {\lambda}} \gamma \lambda \abs{I} = 2C \gamma \abs{I}
\end{align*}
\end{proof}

As a consequence of good-$\lambda$ inequality, we prove the following theorem.
\begin{theorem}{\label{Sak_lem_7.10}}
Let $\seq{a}$ be a nonnegative sequence in $\ell^p(\ZZ), 1 < p < \infty$.Then
\[ \sum_{m \in \ZZ} \abs{M_d a(m)}^p   \leq C \sum_{m \in \ZZ} \abs{M^{\#}a(m)}^p  \]
where $M_d$ is the dyadic maximal operator and $M^{\#}$ is the sharp maximal operator, whenever, the left hand side is finite.
\end{theorem}
\begin{proof}
For a positive integer $N>0$, let
\[ I_N  = \int_0^N p \lambda^{p-1} \abs{  \left \{ m \in \ZZ: M_d a(m) > \lambda \right \} } d\lambda   \]
$I_N$ is finite, since $ a \in \ell^{p}(\ZZ) $ implies $M_d a \in \ell^{p}(\ZZ) $ 
\begin{align*}
 I_N & = \int_0^N p \lambda^{p-1} \abs{  \left \{ m \in \ZZ: M_d a(m) > \lambda \right \} } d\lambda  \\
&= 2^p \int_0^{\frac{N}{2}}  p \lambda^{p-1} \abs{  \left \{ m \in \ZZ: M_d a(m) > 2\lambda \right \} }  d\lambda\\
& \leq 2^p \int_0^{\frac{N}{2}}  p \lambda^{p-1} \abs{  \left \{ m \in \ZZ: M_d a(m) > 2\lambda, M^\# a(m) \leq \gamma \lambda \right \} } d\lambda  +  \\
& 2^p \int_0^{\frac{N}{2}}  p \lambda^{p-1} \abs{  \left \{ m \in \ZZ: M_d a(m) > 2\lambda, M^\# a(m) > \gamma \lambda \right \} } d\lambda  \\
& \leq 2^p \int_0^{\frac{N}{2}}  p \lambda^{p-1} C {\gamma}  \abs{  \left \{ m \in \ZZ: M_d a(m) > \lambda  \right \} } d\lambda  +  \\
& 2^p \int_0^{\frac{N}{2}}  p \lambda^{p-1} \abs{  \left \{ m \in \ZZ: M^\# a(m) > \gamma \lambda \right \} } d\lambda \\
& \leq 2^p C {\gamma} \int_0^N  p \lambda^{p-1}   \abs{  \left \{ m \in \ZZ: M_d a(m) > \lambda  \right \} } d\lambda \quad +  \\
& 2^p \int_0^{\frac{N}{2}}  p \lambda^{p-1} \abs{  \left \{ m \in \ZZ: M^\# a(m) > \gamma \lambda \right \} }  d\lambda
\end{align*}
It follows that
\[ 
(1 - 2^p C {\gamma})I_N \leq  2^p \int_0^{\frac{N}{2}}  p \lambda^{p-1} \abs{  \left \{ m \in \ZZ: M^\# a(m) > \gamma \lambda \right \} } d\lambda \]
Now choose $\gamma = \frac{1}{C2^{p+1}}$ such that $(1 - 2^p C {\gamma}) = \frac{1}{2}$. Then,
\begin{align*}
\frac{1}{2}I_N \leq 2^p \int_0^{\frac{N}{2}}  p \lambda^{p-1} \abs{  \left \{ m \in \ZZ: M^\# a(m) > \gamma \lambda \right \} }   d\lambda \\
\leq \frac{2^p}{\gamma^p}  \int_0^{\frac{N}{2}}  p \lambda^{p-1} \abs{  \left \{ m \in \ZZ: M^\# a(m) > \lambda \right \} }  d\lambda
\end{align*}
Now, take $N \to \infty$,  we get

\[ \sum_{m \in \ZZ} M_d {a(m)}^p  \leq C \sum_{ m \in \ZZ} {M^{\#}a(m)}^p  \]

\end{proof}

%
%

\label{'ubl'}  
\end{document}